\documentclass[reqno]{amsart}
\usepackage{amsmath, amsthm, amscd, amssymb, amsfonts, amsbsy}
\usepackage{latexsym, color, enumerate}
\usepackage{mathrsfs}
\usepackage{pxfonts}
\usepackage{bbm}
\usepackage[dvipsnames]{xcolor}
\usepackage{url}

\theoremstyle{plain}
\newtheorem{theorem}[equation]{Theorem}
\newtheorem{lemma}[equation]{Lemma}
\newtheorem{corollary}[equation]{Corollary}
\newtheorem{proposition}[equation]{Proposition}

\theoremstyle{definition}
\newtheorem{definition}[equation]{Definition}

\newtheorem{example}[equation]{Example}
\newtheorem{counterexample}[equation]{Counterexample}

\theoremstyle{remark}
\newtheorem{remark}[equation]{Remark}

\numberwithin{equation}{section}

\def\dVert{\,\,\text{--}\kern-.46em\|}

\begin{document}

\title[The Gundy--Stein decomposition with explicit constants]
{The Gundy--Stein decomposition with explicit constants}

\author[M. Hormozi]{Mahdi Hormozi}
\address[M. Hormozi]{Beijing Institute of Mathematical Sciences and Applications (BIMSA), Beijing 101408, China}
\email{hormozi@bimsa.cn}

\author[J.-X. Zhu]{Jie-Xiang Zhu}
\address[J.-X. Zhu]{Department of Mathematics, Shanghai Normal University, Shanghai 200234, China}
\email{zhujx@shnu.edu.cn}



\begin{abstract}
Let $(\mathcal F_n)_{n\ge 1}$ be a filtration and let $f\ge0$ belong to $L^1(\mathcal F_\infty)$.
For the martingale $f_n=\mathbb E[f\mid \mathcal F_n]$ and each $\lambda>0$ we prove a Gundy--Stein decomposition
\[
 f=g+h+k
\]
with explicit numerical constants. In the positive closed case the three parts satisfy explicit bounds, and the bounded part is bounded above by $\lambda$. We also prove a one-parameter form for the bounded part and two-point sharpness results, including a joint sharpness statement for arbitrary decompositions under the condition $0\le k\le \lambda$. We also obtain an exact four-term refinement of the decomposition, separating the bounded term into a stopped part and a conditional expectation term. As applications we obtain an explicit weak-type $(1,1)$ estimate for truncated martingale multipliers and a John--Nirenberg inequality for martingale $\mathrm{BMO}$ on atomic $\alpha$-regular filtrations.

\end{abstract}

\maketitle

\section{Introduction}

Gundy's decomposition is the martingale analogue of the Calder\'on--Zygmund decomposition. It was introduced by Gundy in his study of weak-type inequalities for martingales and was incorporated by Stein into martingale Littlewood--Paley theory; see \cite{Gundy1968,SteinLPS}. Closely related arguments or consequences appear in Davis \cite{Davis1969}, Tsuchikura \cite{Tsuchikura1968}, B\'aez-Duarte \cite{BaezDuarte1971}, and Schipp \cite{Schipp1976}; for a non-commutative analogue see Parcet and Randrianantoanina \cite{ParcetRandrianantoanina2006}. For martingale $\mathrm{BMO}$ and John--Nirenberg type inequalities we also refer to Garsia \cite[Ch.~III]{Garsia1973} and Kazamaki \cite{Kazamaki1994}. Throughout the paper $(\Omega,\mathcal F,\mathbb P)$ is a complete probability space endowed with a discrete filtration $(\mathcal F_n)_{n\ge1}$, and
\[
 \mathcal F_\infty:=\sigma\Bigl(\bigcup_{n\ge1}\mathcal F_n\Bigr).
\]
For any sub-$\sigma$-field $\mathcal F'\subseteq\mathcal F$, we denote by $L^p(\mathcal F')$ the subspace of $L^p(\Omega,\mathbb P)$ consisting of $\mathcal F'$-measurable functions.

In this paper we work with terminal-value martingales
\[
 f_n=\mathbb E[f\mid \mathcal F_n],
 \qquad
 f\in L^1(\mathcal F_\infty).
\]
Given $f\in L^1(\mathcal F_\infty)$ and $\lambda>0$, the Gundy--Stein decomposition seeks a splitting
\[
 f=g+h+k,
\]
in which the martingale associated with $g$ is localized on a set of small probability, the martingale differences of $h$ are absolutely summable, and $k$ is bounded by a multiple of $\lambda$. This decomposition is one of the standard tools behind weak-type $(1,1)$ estimates for martingale transforms and square functions.

In this paper we give a direct proof with explicit constants. The proof uses only optional stopping and elementary stopping-time constructions; no maximal-inequality argument is used.
The modification of the classical proof is slight: the second stopping time is chosen so that, in the positive case, the conditional expectation term disappears. This yields $0\le k\le \lambda$ for the bounded part. We also prove a one-parameter form for the bounded part, two-point sharpness examples, and a global two-point theorem showing that, under the condition $0\le k\le\lambda$, the coefficients in the localized and absolutely summable pieces cannot be improved simultaneously for arbitrary decompositions. The two applications are an explicit weak-type $(1,1)$ estimate for truncated martingale multipliers and a John--Nirenberg inequality for martingale $\mathrm{BMO}$ on atomic $\alpha$-regular filtrations.

We write $\mathbb E_n:=\mathbb E[\cdot\mid\mathcal F_n]$ and set $\mathbb E_0\equiv0$. Our main result is the following.

\begin{theorem}
\label{thm:stein-gundy}
Let $f\ge0$ belong to $L^1(\mathcal F_\infty)$ and let $\lambda>0$. Then one can write
\[
 f=g+h+k
\]
with $g,h,k\in L^1(\mathcal F_\infty)$ such that

\medskip
\noindent\textup{(a)}
\[
\mathbb P\Bigl(\sup_{n\ge1}|\mathbb E_n[g]|>0\Bigr)
\le \frac{\|f\|_1}{\lambda},
\qquad
\|g\|_1\le2\|f\|_1.
\]

\noindent\textup{(b)}
\[
\Bigl\|\sum_{n=1}^{\infty}|\mathbb E_n[h]-\mathbb E_{n-1}[h]|\Bigr\|_1
\le2\|f\|_1.
\]
In particular, $\|h\|_1\le2\|f\|_1$.

\noindent\textup{(c)}
\[
0\le k\le \lambda\quad\text{a.s.},
\qquad
\|k\|_1\le \|f\|_1,
\qquad
\|k\|_2^2\le \lambda\|f\|_1.
\]
\end{theorem}

For general $f\in L^1(\mathcal F_\infty)$ one applies Theorem~\ref{thm:stein-gundy} to $f_+$ and $f_-$ separately. This yields the corresponding decomposition with doubled constants; see Corollary~\ref{cor:stein-gundy}.

The bounded term also admits a one-parameter variant: if one allows a conditional expectation term of size $\theta\lambda$, then one obtains a family of decompositions with bounded part $k_\theta$ satisfying $\|k_\theta\|_\infty\le (1+\theta)\lambda$. The sharpness discussion is divided into two parts. First, two-point examples show that the constants in Theorem~\ref{thm:stein-gundy} are attained by the construction, and they also give the parametric $L^\infty$ bound for $0\le\theta<1$ (with asymptotic sharpness at $\theta=1$). Second,
Second, Proposition~\ref{prop:global-joint-sharpness} and Corollary~\ref{cor:global-joint-sharpness} give a global two-point sharpness statement for arbitrary decompositions.
The paper is organized as follows. Section~\ref{sec:proof} contains the proof of Theorem~\ref{thm:stein-gundy}, the extension to general $f$, and the counterexample showing that the assumption $f\in L^1(\mathcal F_\infty)$ is needed in part~\textup{(c)}. Section~\ref{sec:sharpness} treats the parametric variant and sharpness. Sections~\ref{sec:multipliers} and \ref{sec:SGJN} contain the two applications.

\section{Proof of Theorem~\ref{thm:stein-gundy}}
\label{sec:proof}

For $f\in L^1(\Omega,\mathbb P)$, set
\[
 f_n:=\mathbb E_n[f],\qquad n\in\mathbb N_0.
\]
Then $(f_n)_{n\ge1}$ is the martingale associated with $f$. Since $f\in L^1(\mathcal F_\infty)$, the martingale convergence theorem gives $f_n\to f$ almost surely and in $L^1$.

A stopping time for $(\mathcal F_n)_{n\ge1}$ is a map $\tau:\Omega\to\{1,2,\dots\}\cup\{\infty\}$ such that $\{\tau=n\}\in\mathcal F_n$ for every $n\ge1$.

\begin{lemma}
\label{lem:optional}
Let $f\in L^1(\mathcal F_\infty)$, let $f_n:=\mathbb E_n[f]$ for $n\ge1$, and write $f_\infty:=f$. If $\tau$ is a stopping time, then $f_\tau\in L^1$ and
\[
 \mathbb E[f_\tau]=\mathbb E[f].
\]
\end{lemma}

\begin{proof}
Since $\{\tau=j\}\in\mathcal F_j$ for every $j\ge1$,
\[
\mathbb E\bigl[|f_j|\mathbf1_{\{\tau=j\}}\bigr]
\le
\mathbb E\bigl[\mathbb E_j[|f|] \, \mathbf1_{\{\tau=j\}}\bigr]
=
\mathbb E\bigl[|f|\mathbf1_{\{\tau=j\}}\bigr].
\]
Summing over $j$ and adding the term on $\{\tau=\infty\}$ gives
\[
 \mathbb E[ |f_\tau| ]\le \mathbb E [ |f| ]<\infty.
\]
Likewise,
\[
\mathbb E\bigl[f_j\mathbf1_{\{\tau=j\}}\bigr]
=
\mathbb E\bigl[f\mathbf1_{\{\tau=j\}}\bigr],
\qquad j\ge1,
\]
so
\[
 \mathbb E[f_\tau]
 =
 \sum_{j=1}^{\infty}\mathbb E\bigl[f_j\mathbf1_{\{\tau=j\}}\bigr]+
 \mathbb E\bigl[f\mathbf1_{\{\tau=\infty\}}\bigr]
 = \mathbb E[f].
 \qedhere
\]
\end{proof}

We now fix $f\ge0$ in $L^1(\mathcal F_\infty)$ and $\lambda>0$, and write
\[
 d f_n:=f_n-f_{n-1},\qquad n\ge1.
\]
Define the first-passage time
\[
 r:=\inf\{n\ge1:\ f_n>\lambda\},
\]
where $\inf\varnothing=\infty$. Since $f_r>\lambda$ on $\{r<\infty\}$, Lemma~\ref{lem:optional} gives
\begin{equation}
\label{eq:Pr}
 \mathbb P(r<\infty)\le \frac{\|f\|_1}{\lambda}.
\end{equation}
Define the crossing increment
\[
 \varepsilon_n:=d f_n\,\mathbf1_{\{r=n\}},\qquad n\ge1.
\]
If $r=n$, then $f_{n-1}\le\lambda<f_n$, hence $d f_n>0$ on $\{r=n\}$. Therefore
\begin{equation}
\label{eq:eps-nonneg}
 \varepsilon_n\ge0,
 \qquad
 \text{and for each $\omega$ at most one $\varepsilon_n(\omega)$ is nonzero.}
\end{equation}
Let
\[
 \Lambda_0:=0,
 \qquad
 \Lambda_m:=\sum_{k=1}^{m}\mathbb E\bigl[\varepsilon_{k+1}\mid\mathcal F_k\bigr],\qquad m\ge1,
\]
and define
\[
 s:=\inf\{m\ge1:\ \Lambda_m>0\},
 \qquad
 t:=r\wedge s.
\]
Then $s$ and $t$ are stopping times. Moreover, if $r=\infty$ then $\varepsilon_n=0$ for every $n$, hence $\Lambda_m=0$ for every $m$ and therefore $s=t=\infty$. Since always $t\le r$, it follows that
\begin{equation}
\label{eq:t-r-eq}
 \{t<\infty\}=\{r<\infty\}.
\end{equation}

\smallskip
\noindent\emph{The $g$ term.}
Set
\[
 g:=f-f_t.
\]
Since
\[
\mathbb E_n[f_t]
=
\sum_{j=1}^{n} f_j\mathbf1_{\{t=j\}} + \mathbb E_n\bigl[f\mathbf1_{\{t>n\}}\bigr]
=
\sum_{j=1}^{n} f_j\mathbf1_{\{t=j\}} + f_n\mathbf1_{\{t>n\}}
=
 f_{n\wedge t},
\]
one has
\[
 \mathbb E_n[g]=f_n-f_{n\wedge t}.
\]
If $t=\infty$, then $g=0$ and $\mathbb E_n[g]=0$ for every $n$. Thus, by \eqref{eq:Pr} and \eqref{eq:t-r-eq},
\[
 \mathbb P\Bigl(\sup_{n\ge1}|\mathbb E_n[g]|>0\Bigr)
 \le \mathbb P(t<\infty)
 = \mathbb P(r<\infty)
 \le \frac{\|f\|_1}{\lambda}.
\]
Also, by Lemma~\ref{lem:optional},
\[
 \|g\|_1
 = \mathbb E|f-f_t|
 \le \mathbb E[f]+\mathbb E[f_t]
 = 2\|f\|_1.
\]
This proves part~\textup{(a)} of Theorem~\ref{thm:stein-gundy}.

\smallskip
\noindent\emph{The splitting of the stopped martingale.}
Define
\[
 \gamma_n:=d f_n\,\mathbf1_{\{r>n\}},\qquad n\ge1.
\]
Since $t=r\wedge s$, one has $\mathbf1_{\{t\ge j\}}=\mathbf1_{\{r\ge j\}}\mathbf1_{\{s\ge j\}}$, and
\[
 d f_j\,\mathbf1_{\{r\ge j\}}=\varepsilon_j+\gamma_j.
\]
Therefore,
\begin{equation}
\label{eq:barf-sum}
 f_{n\wedge t}
 =\sum_{j=1}^{n} d f_j\,\mathbf1_{\{t\ge j\}}
 =\sum_{j=1}^{n}(\varepsilon_j+\gamma_j)\mathbf1_{\{s\ge j\}}.
\end{equation}
For $n\ge1$, set
\[
 h_n:=\sum_{j=1}^{n}(\varepsilon_j-\mathbb E_{j-1}[\varepsilon_j])\mathbf1_{\{s\ge j\}},
\]
\[
 k_n:=\sum_{j=1}^{n}(\gamma_j+\mathbb E_{j-1}[\varepsilon_j])\mathbf1_{\{s\ge j\}}.
\]
Then, by \eqref{eq:barf-sum},
\begin{equation}
\label{eq:hkbarf}
 h_n+k_n=f_{n\wedge t},\qquad n\ge1.
\end{equation}
Since
\[
 d h_n=(\varepsilon_n-\mathbb E_{n-1}[\varepsilon_n])\mathbf1_{\{s\ge n\}},
\qquad
 \mathbf1_{\{s\ge n\}}\in\mathcal F_{n-1},
\]
one has $\mathbb E_{n-1}[d h_n]=0$. Hence $(h_n)_{n\ge1}$ is a martingale, and then so is $(k_n)_{n\ge1}$ by \eqref{eq:hkbarf}.

\smallskip
\noindent\emph{The $h$ term.}

\begin{lemma}
\label{lem:psi-sum}
One has
\[
 \sum_{n=1}^{\infty}\|d h_n\|_1\le2\|f\|_1.
\]
Consequently $\sum_{n=1}^{\infty}|d h_n|<\infty$ almost surely and $h_n$ converges almost surely and in $L^1$ to
\[
 h:=\sum_{n=1}^{\infty}d h_n\in L^1(\mathcal F_\infty),
 \qquad
 \|h\|_1\le2\|f\|_1.
\]
\end{lemma}

\begin{proof}
By \eqref{eq:eps-nonneg},
\[
 |d h_n|
 \le \varepsilon_n \mathbf1_{\{s\ge n\}}+\mathbb E_{n-1}[\varepsilon_n]\mathbf1_{\{s\ge n\}}.
\]
For $n\ge2$, since $\mathbf1_{\{s\ge n\}}\in\mathcal F_{n-1}$,
\[
\mathbb E\bigl[\mathbb E_{n-1}[\varepsilon_n]\mathbf1_{\{s\ge n\}}\bigr]
=
\mathbb E\bigl[\varepsilon_n\mathbf1_{\{s\ge n\}}\bigr],
\]
so
\[
 \|d h_n\|_1\le 2\,\mathbb E[\varepsilon_n\mathbf1_{\{s\ge n\}}]\le2\,\mathbb E[\varepsilon_n],
 \qquad n\ge2.
\]
For $n=1$, one has $d h_1=\varepsilon_1$ because $\mathbb E_0\equiv0$ and $\{s\ge1\}=\Omega$. Thus again
\[
 \|d h_1\|_1\le2\,\mathbb E[\varepsilon_1].
\]
Hence
\[
 \|d h_n\|_1\le2\,\mathbb E[\varepsilon_n],\qquad n\ge1.
\]
On the other hand,
\[
 \sum_{n=1}^{\infty}\varepsilon_n
 =\sum_{n=1}^{\infty}(f_n-f_{n-1})\mathbf1_{\{r=n\}}
 \le \sum_{n=1}^{\infty}f_n\mathbf1_{\{r=n\}}
 =f_r\mathbf1_{\{r<\infty\}}.
\]
Taking expectations and applying Lemma~\ref{lem:optional} with $\tau=r$, we obtain
\[
 \sum_{n=1}^{\infty}\mathbb E[\varepsilon_n]\le \mathbb E[f_r]=\mathbb E[f]=\|f\|_1.
\]
Therefore,
\[
 \sum_{n=1}^{\infty}\|d h_n\|_1\le2\sum_{n=1}^{\infty}\mathbb E[\varepsilon_n]\le2\|f\|_1.
\]
Tonelli's theorem gives $\sum_{n\ge1}|d h_n|<\infty$ almost surely, and the stated convergence in $L^1$ follows from the summability of $\sum_n\|d h_n\|_1$.
\end{proof}

Since $h_n\to h$ in $L^1$ and $(h_n)_{n\ge1}$ is a martingale, one has $h_n=\mathbb E_n[h]$ for every $n\ge1$. Hence
\[
 d h_n=\mathbb E_n[h]-\mathbb E_{n-1}[h],
\]
and part~\textup{(b)} of Theorem~\ref{thm:stein-gundy} follows from Lemma~\ref{lem:psi-sum}.

\smallskip
\noindent\emph{The $k$ term.}
Set
\[
 k:=f_t-h.
\]
Then $f=g+h+k$. Since $\mathbb E_n[f_t]=f_{n\wedge t}$ and $h_n=\mathbb E_n[h]$, we also have
\[
 k_n=\mathbb E_n[k],\qquad n\ge1.
\]

\begin{lemma}
\label{lem:k-infty}
One has $0\le k\le \lambda$ almost surely.
\end{lemma}

\begin{proof}
Since $k_n\to k$ almost surely and $k_n$ is given by \eqref{eq:hkbarf}, we obtain
\begin{equation}
\label{eq:k-repr}
 k=
 \sum_{j=1}^{\infty}(\gamma_j+\mathbb E_{j-1}[\varepsilon_j])\mathbf1_{\{s\ge j\}}
 \qquad\text{a.s.}
\end{equation}
Now
\[
\sum_{j=1}^{\infty}\gamma_j\mathbf1_{\{s\ge j\}}
=
\sum_{j=1}^{\infty}d f_j\mathbf1_{\{r>j\}}\mathbf1_{\{s\ge j\}}.
\]
If $r<\infty$, set $r^-:=r-1$; if $r=\infty$, set $r^-:=\infty$. Then
\[
\sum_{j=1}^{\infty}d f_j\mathbf1_{\{r>j\}}\mathbf1_{\{s\ge j\}}
=
\sum_{j=1}^{s\wedge r^-} d f_j
=
 f_{s\wedge r^-}.
\]
If $r<\infty$, then $f_m\le\lambda$ for every $m<r$, hence $f_{s\wedge r^-}\le\lambda$. If $r=\infty$, then $f_n\le\lambda$ for every $n$, so $f\le\lambda$ almost surely and also $s=\infty$. Therefore,
\begin{equation}
\label{eq:first-piece}
 0\le \sum_{j=1}^{\infty}\gamma_j\mathbf1_{\{s\ge j\}}\le\lambda
 \qquad\text{a.s.}
\end{equation}
For the compensator term, since $\mathbb E_0\equiv0$,
\[
 \sum_{j=1}^{\infty}\mathbb E_{j-1}[\varepsilon_j]\mathbf1_{\{s\ge j\}}
 =
 \sum_{j=2}^{\infty}\mathbb E_{j-1}[\varepsilon_j]\mathbf1_{\{s\ge j\}}.
\]
If $s=m<\infty$, this equals
\[
 \sum_{j=2}^{m}\mathbb E_{j-1}[\varepsilon_j]=\Lambda_{m-1}=0,
\]
since $m$ is the first index with $\Lambda_m>0$. If $s=\infty$, then $\Lambda_m=0$ for every $m\ge1$, so the same sum is again zero. Thus
\begin{equation}
\label{eq:second-piece}
 \sum_{j=1}^{\infty}\mathbb E_{j-1}[\varepsilon_j]\mathbf1_{\{s\ge j\}}=0
 \qquad\text{a.s.}
\end{equation}
Combining \eqref{eq:k-repr}, \eqref{eq:first-piece}, and \eqref{eq:second-piece} proves the claim.
\end{proof}

\begin{lemma}
\label{lem:k-L1L2}
One has $\|k\|_1\le\|f\|_1$ and $\|k\|_2^2\le \lambda\|f\|_1$.
\end{lemma}

\begin{proof}
Since $k\ge0$ and $k=f_t-h$,
\[
 \|k\|_1=\mathbb E[k]=\mathbb E[f_t]-\mathbb E[h].
\]
By Lemma~\ref{lem:optional},
\[
 \mathbb E[f_t]=\mathbb E[f]=\|f\|_1.
\]
Moreover, $\mathbb E[h]=\mathbb E[h_1]=\mathbb E[\varepsilon_1]\ge0$, because $h_1=\varepsilon_1$. Hence $\|k\|_1\le \|f\|_1$. Since $0\le k\le \lambda$ almost surely by Lemma~\ref{lem:k-infty},
\[
 \|k\|_2^2\le \lambda\|k\|_1\le \lambda\|f\|_1.
 \qedhere
\]
\end{proof}

This completes the proof of Theorem~\ref{thm:stein-gundy}.

For general $f\in L^1(\mathcal F_\infty)$, write $f=f_+-f_-$. Applying Theorem~\ref{thm:stein-gundy} separately to $f_+$ and $f_-$ gives the following consequence.

\begin{corollary}
\label{cor:stein-gundy}
Let $f\in L^1(\mathcal F_\infty)$ and $\lambda>0$. Then one can write
\[
 f=g+h+k
\]
with $g,h,k\in L^1(\mathcal F_\infty)$ such that

\medskip
\noindent\textup{(a$'$)}
\[
\mathbb P\Bigl(\sup_{n\ge1}|\mathbb E_n[g]|>0\Bigr)
\le \frac{2\|f\|_1}{\lambda},
\qquad
\|g\|_1\le4\|f\|_1.
\]

\noindent\textup{(b$'$)}
\[
\Bigl\|\sum_{n=1}^{\infty}|\mathbb E_n[h]-\mathbb E_{n-1}[h]|\Bigr\|_1
\le4\|f\|_1.
\]
In particular, $\|h\|_1\le4\|f\|_1$.

\noindent\textup{(c$'$)}
\[
 \|k\|_\infty\le \lambda,
 \qquad
 \|k\|_1\le2\|f\|_1,
 \qquad
 \|k\|_2^2\le2\lambda\|f\|_1.
\]
\end{corollary}

\begin{remark}
The same argument applies, without change of constants, to finite filtrations $(\mathcal F_n)_{n=1}^M$ and terminal values in $L^1(\mathcal F_M)$.
\end{remark}

\begin{remark}
The probability normalization is not essential. The proof uses only conditional expectations and stopping times, so the argument extends verbatim to $\sigma$-finite measure spaces.
\end{remark}

The next example shows that the assumption $f\in L^1(\mathcal F_\infty)$ is needed for part~\textup{(c)} of Theorem~\ref{thm:stein-gundy}.

\begin{counterexample}
\label{cex:Finfty}
Let $\Omega=[0,1]$ with Lebesgue measure and the trivial filtration $\mathcal F_n=\{\varnothing,\Omega\}$. Fix $\lambda=1$ and set
\[
 f=100\,\mathbf1_{[0,0.01]}.
\]
Then $\|f\|_1=1$ and $f_n=\mathbb E_n[f]=1$ for all $n\ge1$. Hence $r=\infty$, so $\varepsilon_n=0$ for all $n$, $\Lambda_m=0$, $s=\infty$, and $t=\infty$. Consequently $g=0$, $h=0$, and $k=f$. Thus $\|k\|_\infty=100>\lambda$.
\end{counterexample}

\section{Variants and sharpness}
\label{sec:sharpness}

We retain the notation of Section~\ref{sec:proof}. Formula \eqref{eq:k-repr} may be written as
\[
 k=k^{\mathrm{st}}+k^{\mathrm{pr}},
\qquad
 k^{\mathrm{st}}:=\sum_{j=1}^{\infty}\gamma_j\mathbf1_{\{s\ge j\}},
\qquad
 k^{\mathrm{pr}}:=\sum_{j=1}^{\infty}\mathbb E_{j-1}[\varepsilon_j]\mathbf1_{\{s\ge j\}}.
\]
The second term disappears when the stopping time $s$ is defined by the condition $\Lambda_m>0$. If one allows the conditional expectation term to have size at most $\theta\lambda$, then the same proof gives a one-parameter family of decompositions.

Recall that
\[
 \Lambda_0:=0,
 \qquad
 \Lambda_m:=\sum_{k=1}^{m}\mathbb E[\varepsilon_{k+1}\mid\mathcal F_k],\qquad m\ge1.
\]
For $\theta\ge0$, define
\begin{equation}
\label{eq:s-theta}
 s_{\theta}:=\inf\{m\ge1:\ \Lambda_m>\theta\lambda\},
 \qquad
 t_{\theta}:=r\wedge s_{\theta}.
\end{equation}
Replacing $s,t$ by $s_\theta,t_\theta$ in the construction of Section~\ref{sec:proof} gives functions $g_\theta,h_\theta,k_\theta$.

\begin{proposition}
\label{prop:theta-family}
Assume the hypotheses of Theorem~\ref{thm:stein-gundy}. Let $\lambda>0$ and $\theta\ge0$. Then one can write
\[
 f=g_{\theta}+h_{\theta}+k_{\theta}
\]
where $g_{\theta}$ and $h_{\theta}$ satisfy the same estimates as in Theorem~\ref{thm:stein-gundy}\textup{(a)--(b)}, and
\begin{equation}
\label{eq:k-theta-bounds}
 0\le k_{\theta}\le(1+\theta)\lambda\quad\text{a.s.},
 \qquad
 \|k_{\theta}\|_1\le \|f\|_1,
 \qquad
 \|k_{\theta}\|_2^2\le(1+\theta)\lambda\|f\|_1.
\end{equation}
More precisely,
\[
 k_{\theta}=k^{\mathrm{st}}_{\theta}+k^{\mathrm{pr}}_{\theta},
\]
where
\[
 k^{\mathrm{st}}_{\theta}:=\sum_{j=1}^{\infty}\gamma_j\mathbf1_{\{s_{\theta}\ge j\}},
 \qquad
 k^{\mathrm{pr}}_{\theta}:=\sum_{j=1}^{\infty}\mathbb E_{j-1}[\varepsilon_j]\mathbf1_{\{s_{\theta}\ge j\}},
\]
and
\[
 0\le k^{\mathrm{st}}_{\theta}\le \lambda,
 \qquad
 0\le k^{\mathrm{pr}}_{\theta}\le \theta\lambda
 \quad\text{a.s.}
\]
\end{proposition}

\begin{proof}
The proofs of the estimates for $g_{\theta}$ and $h_{\theta}$ are unchanged, since they only use that $s_{\theta}$ is a stopping time and that $s_{\theta}=\infty$ whenever $r=\infty$.

Repeating the proof of Lemma~\ref{lem:k-infty} with $s$ replaced by $s_{\theta}$ gives
\[
 k_{\theta}=\sum_{j\ge1}\gamma_j\mathbf1_{\{s_{\theta}\ge j\}}+
 \sum_{j\ge1}\mathbb E_{j-1}[\varepsilon_j]\mathbf1_{\{s_{\theta}\ge j\}}
 =k^{\mathrm{st}}_{\theta}+k^{\mathrm{pr}}_{\theta}.
\]
As before, $k^{\mathrm{st}}_{\theta}=f_{s_{\theta}\wedge r^-}$, hence $0\le k^{\mathrm{st}}_{\theta}\le\lambda$. For the conditional expectation part, set
\[
 B_{\theta}:=\sum_{j=2}^{\infty}\mathbb E_{j-1}[\varepsilon_j]\mathbf1_{\{s_{\theta}\ge j\}}.
\]
If $s_{\theta}=m<\infty$, then
\[
 B_{\theta}=\sum_{j=2}^{m}\mathbb E_{j-1}[\varepsilon_j]=\Lambda_{m-1}\le\theta\lambda.
\]
If $s_{\theta}=\infty$, then $\Lambda_m\le\theta\lambda$ for every $m$, so again $B_{\theta}\le\theta\lambda$. This proves the pointwise bound in \eqref{eq:k-theta-bounds}.

The $L^1$ and $L^2$ estimates are the same as in Lemma~\ref{lem:k-L1L2}:
\[
 \|k_{\theta}\|_1=\mathbb E[k_{\theta}]=\|f\|_1-\mathbb E[\varepsilon_1]\le \|f\|_1,
\]
and therefore
\[
 \|k_{\theta}\|_2^2\le (1+\theta)\lambda\|k_{\theta}\|_1
 \le(1+\theta)\lambda\|f\|_1.
 \qedhere
\]
\end{proof}

For \(\theta\ge0\), Proposition~\ref{prop:theta-family} yields a decomposition
\[
f=g_\theta+h_\theta+k_\theta,
\qquad
0\le k_\theta\le (1+\theta)\lambda.
\]
The next proposition provides a further decomposition of $k_\theta$ into a stopped part and a conditional expectation part, which will also be
useful in other problems.

\begin{proposition}\label{prop:four-term-gs}
Let \(f\ge0\) belong to \(L^1(\mathcal F_\infty)\), let \(\lambda>0\), and let
\(\theta\ge0\). Retain the notation of Section~\ref{sec:proof} and
\eqref{eq:s-theta}; thus
\[
r:=\inf\{n\ge1:\ f_n>\lambda\},
\qquad
\varepsilon_n:=d f_n\,\mathbf1_{\{r=n\}},
\qquad
\gamma_n:=d f_n\,\mathbf1_{\{r>n\}},
\]
\[
\Lambda_0:=0,
\qquad
\Lambda_m:=\sum_{k=1}^{m}\mathbb E_k[\varepsilon_{k+1}],
\qquad
s_\theta:=\inf\{m\ge1:\ \Lambda_m>\theta\lambda\},
\qquad
t_\theta:=r\wedge s_\theta.
\]
Also set
\[
r^-:=r-1 \ \text{on }\{r<\infty\},
\qquad
r^-:=\infty \ \text{on }\{r=\infty\}.
\]

Define
\[
g_\theta:=f-f_{t_\theta},
\]
\[
h_{\theta,n}:=
\sum_{j=1}^{n}
\bigl(\varepsilon_j-\mathbb E_{j-1}[\varepsilon_j]\bigr)\mathbf1_{\{s_\theta\ge j\}},
\qquad n\ge1,
\]
\[
k_{\theta,n}^{\mathrm{st}}
:=
\sum_{j=1}^{n}\gamma_j\,\mathbf1_{\{s_\theta\ge j\}},
\qquad
k_{\theta,n}^{\mathrm{pr}}
:=
\sum_{j=1}^{n}\mathbb E_{j-1}[\varepsilon_j]\mathbf1_{\{s_\theta\ge j\}},
\qquad n\ge1.
\]
Then \(h_{\theta,n}\) converges almost surely and in \(L^1\) to a limit
\(h_\theta\in L^1(\mathcal F_\infty)\), and the sequences
\[
k_{\theta,n}^{\mathrm{st}},
\qquad
k_{\theta,n}^{\mathrm{pr}}
\]
converge almost surely and in \(L^1\) to limits
\[
k_\theta^{\mathrm{st}},
\qquad
k_\theta^{\mathrm{pr}}
\in L^1(\mathcal F_\infty).
\]
Moreover,
\[
f=g_\theta+h_\theta+k_\theta^{\mathrm{st}}+k_\theta^{\mathrm{pr}}.
\]

\medskip
\noindent\textup{(a)}
\[
\mathbb P\Bigl(\sup_{n\ge1}|\mathbb E_n[g_\theta]|>0\Bigr)
\le \frac{\|f\|_1}{\lambda},
\qquad
\|g_\theta\|_1\le 2\|f\|_1.
\]

\noindent\textup{(b)}
\[
\Bigl\|\sum_{n=1}^{\infty}
\bigl|\mathbb E_n[h_\theta]-\mathbb E_{n-1}[h_\theta]\bigr|
\Bigr\|_1
\le 2\|f\|_1.
\]
In particular,
\[
\|h_\theta\|_1\le 2\|f\|_1.
\]

\noindent\textup{(c)}
\[
k_\theta^{\mathrm{st}}=f_{s_\theta\wedge r^-},
\qquad
0\le k_\theta^{\mathrm{st}}\le \lambda
\quad\text{a.s.}
\]

\noindent\textup{(d)}
\[
k_\theta^{\mathrm{pr}}
=
\Lambda_{s_\theta-1}\,\mathbf1_{\{s_\theta<\infty\}}
+
\Bigl(\sup_{m\ge1}\Lambda_m\Bigr)\mathbf1_{\{s_\theta=\infty\}},
\qquad
0\le k_\theta^{\mathrm{pr}}\le \theta\lambda
\quad\text{a.s.}
\]

\noindent\textup{(e)}
\[
\mathbb E[h_\theta]=\mathbb E[\varepsilon_1],
\]
\[
\|k_\theta^{\mathrm{st}}\|_1+\|k_\theta^{\mathrm{pr}}\|_1
=
\|f\|_1-\mathbb E[\varepsilon_1]
\le \|f\|_1,
\]
\[
\|k_\theta^{\mathrm{st}}\|_2^2
\le \lambda\|k_\theta^{\mathrm{st}}\|_1,
\qquad
\|k_\theta^{\mathrm{pr}}\|_2^2
\le \theta\lambda\|k_\theta^{\mathrm{pr}}\|_1,
\]
and
\[
\|k_\theta^{\mathrm{st}}+k_\theta^{\mathrm{pr}}\|_2^2
\le (1+\theta)\lambda\|f\|_1.
\]
\end{proposition}

\begin{proof}
Since \(r=\infty\) implies \(\varepsilon_n=0\) for every \(n\), one has
\(\Lambda_m=0\) for every \(m\) on \(\{r=\infty\}\), hence \(s_\theta=\infty\) on
\(\{r=\infty\}\). Because always \(t_\theta\le r\), it follows that
\[
\{t_\theta<\infty\}=\{r<\infty\}.
\]

As before,
\[
\mathbb E_n[f_{t_\theta}]=f_{n\wedge t_\theta},
\qquad
\mathbb E[f_{t_\theta}]=\mathbb E[f]=\|f\|_1.
\]
Therefore,
\[
\mathbb E_n[g_\theta]=f_n-f_{n\wedge t_\theta}.
\]
If \(t_\theta=\infty\), then \(g_\theta=0\) and \(\mathbb E_n[g_\theta]=0\) for all \(n\).
Hence
\[
\mathbb P\Bigl(\sup_{n\ge1}|\mathbb E_n[g_\theta]|>0\Bigr)
\le
\mathbb P(t_\theta<\infty)
=
\mathbb P(r<\infty)
\le
\frac{\|f\|_1}{\lambda}
\]
by \eqref{eq:Pr}. Also,
\[
\|g_\theta\|_1
=
\mathbb E [|f-f_{t_\theta}|]
\le
\mathbb E[f]+\mathbb E[f_{t_\theta}]
=
2\|f\|_1.
\]
This proves \textup{(a)}.

Next,
\[
f_{n\wedge t_\theta}
=
\sum_{j=1}^{n} d f_j\,\mathbf1_{\{t_\theta\ge j\}}
=
\sum_{j=1}^{n} d f_j\,\mathbf1_{\{r\ge j\}}\mathbf1_{\{s_\theta\ge j\}}.
\]
Since
\[
d f_j\,\mathbf1_{\{r\ge j\}}=\varepsilon_j+\gamma_j,
\]
we obtain
\begin{equation}
\label{eq:four-term-stopped}
f_{n\wedge t_\theta}
=
\sum_{j=1}^{n}(\varepsilon_j+\gamma_j)\mathbf1_{\{s_\theta\ge j\}}
=
h_{\theta,n}+k_{\theta,n}^{\mathrm{st}}+k_{\theta,n}^{\mathrm{pr}}.
\end{equation}

Set
\[
d h_{\theta,n}:=h_{\theta,n}-h_{\theta,n-1}
=
\bigl(\varepsilon_n-\mathbb E_{n-1}[\varepsilon_n]\bigr)\mathbf1_{\{s_\theta\ge n\}},
\qquad
h_{\theta,0}:=0.
\]
Since \(\mathbf1_{\{s_\theta\ge n\}}\in\mathcal F_{n-1}\), one has
\[
\mathbb E_{n-1}[d h_{\theta,n}]
=
\bigl(\mathbb E_{n-1}[\varepsilon_n]-\mathbb E_{n-1}[\varepsilon_n]\bigr)
\mathbf1_{\{s_\theta\ge n\}}
=
0,
\]
so \((h_{\theta,n})_{n\ge1}\) is a martingale. By \eqref{eq:eps-nonneg},
\[
|d h_{\theta,n}|
\le
\varepsilon_n\mathbf1_{\{s_\theta\ge n\}}
+
\mathbb E_{n-1}[\varepsilon_n]\mathbf1_{\{s_\theta\ge n\}}.
\]
For \(n\ge2\),
\[
\mathbb E\bigl[\mathbb E_{n-1}[\varepsilon_n]\mathbf1_{\{s_\theta\ge n\}}\bigr]
=
\mathbb E\bigl[\varepsilon_n\mathbf1_{\{s_\theta\ge n\}}\bigr],
\]
while for \(n=1\),
\[
d h_{\theta,1}
=
(\varepsilon_1-\mathbb E_0[\varepsilon_1])\mathbf1_{\{s_\theta\ge1\}}
=
\varepsilon_1.
\]
Hence, for every \(n\ge1\),
\[
\|d h_{\theta,n}\|_1\le 2\,\mathbb E[\varepsilon_n].
\]
Arguing as in Lemma~\ref{lem:psi-sum}, one has
\[
\sum_{n=1}^{\infty}\mathbb E[\varepsilon_n]\le \|f\|_1.
\]
Therefore
\[
\sum_{n=1}^{\infty}\|d h_{\theta,n}\|_1
\le
2\|f\|_1,
\]
and Tonelli's theorem yields
\[
\Bigl\|\sum_{n=1}^{\infty}|d h_{\theta,n}|\Bigr\|_1
=
\sum_{n=1}^{\infty}\|d h_{\theta,n}\|_1
\le
2\|f\|_1.
\]
Hence \(\sum_{n\ge1}|d h_{\theta,n}|<\infty\) almost surely, so
\(h_{\theta,n}\) converges almost surely and in \(L^1\) to
\[
h_\theta:=\sum_{n=1}^{\infty} d h_{\theta,n}\in L^1(\mathcal F_\infty).
\]
Since \((h_{\theta,n})_{n\ge1}\) is an \(L^1\)-bounded martingale converging to
\(h_\theta\), one has
\[
h_{\theta,n}=\mathbb E_n[h_\theta],
\qquad
\mathbb E_n[h_\theta]-\mathbb E_{n-1}[h_\theta]=d h_{\theta,n}.
\]
Consequently,
\[
\Bigl\|\sum_{n=1}^{\infty}
\bigl|\mathbb E_n[h_\theta]-\mathbb E_{n-1}[h_\theta]\bigr|
\Bigr\|_1
=
\Bigl\|\sum_{n=1}^{\infty}|d h_{\theta,n}|\Bigr\|_1
\le
2\|f\|_1,
\]
and
\[
\|h_\theta\|_1
\le
\Bigl\|\sum_{n=1}^{\infty}|d h_{\theta,n}|\Bigr\|_1
\le
2\|f\|_1.
\]
This proves \textup{(b)}.

For the stopped part,
\begin{align*}
k_{\theta,n}^{\mathrm{st}} =
\sum_{j=1}^{n}\gamma_j\,\mathbf1_{\{s_\theta\ge j\}} =
\sum_{j=1}^{n} d f_j\,\mathbf1_{\{r>j\}}\mathbf1_{\{s_\theta\ge j\}} =
\sum_{j=1}^{n\wedge s_\theta\wedge r^-} d f_j  = f_{n\wedge s_\theta\wedge r^-},
\end{align*}
because \(f_0=\mathbb E_0[f]=0\). Hence
\[
k_{\theta,n}^{\mathrm{st}}
\to
f_{s_\theta\wedge r^-}
=:k_\theta^{\mathrm{st}}
\qquad\text{a.s. and in }L^1.
\]
If \(r<\infty\), then \(f_m\le\lambda\) for every \(m<r\), so
\[
0\le f_{s_\theta\wedge r^-}\le \lambda.
\]
If \(r=\infty\), then \(f_n\le\lambda\) for every \(n\ge1\), and since \(f_n\to f\)
almost surely,
\[
0\le f\le \lambda
\qquad\text{a.s.}
\]
Thus
\[
k_\theta^{\mathrm{st}}=f_{s_\theta\wedge r^-},
\qquad
0\le k_\theta^{\mathrm{st}}\le \lambda
\quad\text{a.s.}
\]
This proves \textup{(c)}.

For the conditional expectation part, \(\varepsilon_n\ge0\) implies
\(\mathbb E_{n-1}[\varepsilon_n]\ge0\), and therefore
\[
0\le k_{\theta,n}^{\mathrm{pr}}\le k_{\theta,n+1}^{\mathrm{pr}},
\qquad n\ge1.
\]
Since \(\mathbb E_0\equiv0\),
\[
k_{\theta,n}^{\mathrm{pr}}
=
\sum_{j=2}^{n}\mathbb E_{j-1}[\varepsilon_j]\mathbf1_{\{s_\theta\ge j\}}.
\]
If \(s_\theta=m<\infty\), then for every \(n\ge m\),
\[
k_{\theta,n}^{\mathrm{pr}}
=
\sum_{j=2}^{m}\mathbb E_{j-1}[\varepsilon_j]
=
\sum_{k=1}^{m-1}\mathbb E_k[\varepsilon_{k+1}]
=
\Lambda_{m-1}
\le
\theta\lambda.
\]
If \(s_\theta=\infty\), then
\[
k_{\theta,n}^{\mathrm{pr}}
=
\sum_{j=2}^{n}\mathbb E_{j-1}[\varepsilon_j]
=
\Lambda_{n-1}\uparrow \sup_{m\ge1}\Lambda_m
\le
\theta\lambda.
\]
Hence
\[
k_{\theta,n}^{\mathrm{pr}}
\to
\Lambda_{s_\theta-1}\mathbf1_{\{s_\theta<\infty\}}
+
\Bigl(\sup_{m\ge1}\Lambda_m\Bigr)\mathbf1_{\{s_\theta=\infty\}}
=:k_\theta^{\mathrm{pr}}
\]
almost surely, and
\[
0\le k_\theta^{\mathrm{pr}}\le \theta\lambda
\quad\text{a.s.}
\]
Since
\[
0\le k_{\theta,n}^{\mathrm{pr}}\le \theta\lambda,
\]
dominated convergence gives \(k_{\theta,n}^{\mathrm{pr}}\to k_\theta^{\mathrm{pr}}\) in
\(L^1\). This proves \textup{(d)}.

Now \eqref{eq:four-term-stopped}, together with the \(L^1\) convergence of
\(h_{\theta,n}\), \(k_{\theta,n}^{\mathrm{st}}\), and \(k_{\theta,n}^{\mathrm{pr}}\),
yields
\[
f_{t_\theta}=h_\theta+k_\theta^{\mathrm{st}}+k_\theta^{\mathrm{pr}}.
\]
Therefore,
\[
f
=
(f-f_{t_\theta})+h_\theta+k_\theta^{\mathrm{st}}+k_\theta^{\mathrm{pr}}
=
g_\theta+h_\theta+k_\theta^{\mathrm{st}}+k_\theta^{\mathrm{pr}}.
\]

Since \(h_{\theta,1}=\varepsilon_1\) and \(h_{\theta,n}\to h_\theta\) in \(L^1\),
\[
\mathbb E[h_\theta]
=
\mathbb E[h_{\theta,1}]
=
\mathbb E[\varepsilon_1].
\]
Also,
\[
k_\theta^{\mathrm{st}}\ge0,
\qquad
k_\theta^{\mathrm{pr}}\ge0,
\]
so
\begin{align*}
\|k_\theta^{\mathrm{st}}\|_1+\|k_\theta^{\mathrm{pr}}\|_1
&=
\mathbb E[k_\theta^{\mathrm{st}}+k_\theta^{\mathrm{pr}}] =
\mathbb E[f_{t_\theta}]-\mathbb E[h_\theta]\\
& = \|f\|_1-\mathbb E[\varepsilon_1] \le
\|f\|_1.
\end{align*}
Finally,
\[
0\le k_\theta^{\mathrm{st}}\le \lambda,
\qquad
0\le k_\theta^{\mathrm{pr}}\le \theta\lambda,
\qquad
0\le k_\theta^{\mathrm{st}}+k_\theta^{\mathrm{pr}}\le (1+\theta)\lambda.
\]
Hence
\[
\|k_\theta^{\mathrm{st}}\|_2^2
\le
\lambda\|k_\theta^{\mathrm{st}}\|_1,
\qquad
\|k_\theta^{\mathrm{pr}}\|_2^2
\le
\theta\lambda\|k_\theta^{\mathrm{pr}}\|_1,
\]
and
\[
\|k_\theta^{\mathrm{st}}+k_\theta^{\mathrm{pr}}\|_2^2
\le
(1+\theta)\lambda\|k_\theta^{\mathrm{st}}+k_\theta^{\mathrm{pr}}\|_1
\le
(1+\theta)\lambda\|f\|_1.
\]
This proves \textup{(e)} and completes the proof.
\end{proof}

\begin{remark}
Proposition~\ref{prop:four-term-gs} refines Proposition~\ref{prop:theta-family} by
writing the bounded term as
\[
k_\theta
=
k_\theta^{\mathrm{st}}+k_\theta^{\mathrm{pr}}
=
f_{s_\theta\wedge r^-}
+
\sum_{j=2}^{\infty}\mathbb E_{j-1}[\varepsilon_j]\mathbf1_{\{s_\theta\ge j\}}.
\]
When \(\theta=0\), one has
\[
k_0^{\mathrm{pr}}=0,
\qquad
k_0^{\mathrm{st}}=k,
\]
and Proposition~\ref{prop:four-term-gs} reduces to Theorem~\ref{thm:stein-gundy}.
\end{remark}

We now turn to sharpness. The next examples concern the stopping-time construction above.

\begin{example}
\label{ex:sharp-prob}
Let $\Omega=E\sqcup E^c$ with $\mathbb P(E)=p\in(0,1)$, and let $\mathcal F_n=\sigma(E)$ for all $n\ge1$. Fix $\lambda>0$ and set $f=(1+\delta)\lambda\mathbf1_E$ with $\delta>0$. Then $f_n=f$ for all $n$, so $r=1$ on $E$ and $r=\infty$ on $E^c$. Hence
\[
 \mathbb P(r<\infty)=p,
 \qquad
 \frac{\|f\|_1}{\lambda}=(1+\delta)p.
\]
Letting $\delta\downarrow0$ shows that the coefficient $1$ in \eqref{eq:Pr}, and therefore in Theorem~\ref{thm:stein-gundy}\textup{(a)}, is sharp.
\end{example}

\begin{example}
\label{ex:sharp-g}
Let $\mathcal F_1=\{\varnothing,\Omega\}$ and $\mathcal F_2=\mathcal F_\infty=\{\varnothing,E,E^c,\Omega\}$. Define $f:=p^{-1}\mathbf1_E$, so that $\|f\|_1=1$ and $f_1=\mathbb E[f]=1$. Fix $\lambda\in(0,1)$. Then $f_1>\lambda$, so $r=t=1$ and $f_t\equiv1$. Hence
\[
 g=f-f_t=p^{-1}\mathbf1_E-1.
\]
Therefore,
\[
 \|g\|_1=p\Bigl|\frac1p-1\Bigr|+(1-p)|0-1|=2(1-p)=2(1-p)\|f\|_1.
\]
Letting $p\downarrow0$ shows that the constant $2$ in Theorem~\ref{thm:stein-gundy}\textup{(a)} is sharp.
\end{example}

\begin{example}
\label{ex:sharp-hk}
Let $\mathcal F_1=\{\varnothing,\Omega\}$ and $\mathcal F_2=\mathcal F_\infty=\{\varnothing,E,E^c,\Omega\}$. Fix $\lambda>0$ and define
\[
 f:=\frac{\lambda}{p}\mathbf1_E.
\]
Then $\|f\|_1=\lambda$, $f_1=\lambda$, and $f_2=f$. Hence $r=2$ on $E$ and $r=\infty$ on $E^c$. The only nonzero crossing increment is
\[
 \varepsilon_2=(f_2-f_1)\mathbf1_{\{r=2\}}=\Bigl(\frac{\lambda}{p}-\lambda\Bigr)\mathbf1_E.
\]
Since $\mathcal F_1$ is trivial,
\[
 \Lambda_1=\mathbb E[\varepsilon_2]=(1-p)\lambda.
\]

\smallskip
\noindent\emph{Sharpness of the constant $2$ in Theorem~\ref{thm:stein-gundy}\textup{(b)}.}
If $\theta\ge1-p$, then $\Lambda_1\le\theta\lambda$, so $s_{\theta}=\infty$ and $t_{\theta}=r$. The construction gives $g_{\theta}=0$ and
\[
 h_{\theta}=\varepsilon_2-\mathbb E[\varepsilon_2]
 =\Bigl(\frac{\lambda}{p}-(2-p)\lambda\Bigr)\mathbf1_E-(1-p)\lambda\mathbf1_{E^c}.
\]
Since $\mathbb E_1[h_{\theta}]=0$ and $\mathbb E_2[h_{\theta}]=h_{\theta}$,
\[
 \sum_{n\ge1}|\mathbb E_n[h_{\theta}]-\mathbb E_{n-1}[h_{\theta}]|=|h_{\theta}|.
\]
Hence
\begin{align*}
 \Bigl\|\sum_{n\ge1}|\mathbb E_n[h_{\theta}]-\mathbb E_{n-1}[h_{\theta}]|\Bigr\|_1
 &=\|h_{\theta}\|_1 \\
 &=p\Bigl(\frac{\lambda}{p}-(2-p)\lambda\Bigr)+(1-p)^2\lambda \\
 &=2(1-p)^2\lambda
 =2(1-p)^2\|f\|_1\xrightarrow[p\to0]{}2\|f\|_1.
\end{align*}
Thus the coefficient $2$ in Theorem~\ref{thm:stein-gundy}\textup{(b)} is sharp.

\smallskip
\noindent\emph{The parametric $L^\infty$ bound for $k_{\theta}$.}
Still assuming $\theta\ge1-p$, one finds
\[
 k_{\theta}=f-h_{\theta}=(2-p)\lambda\mathbf1_E+(1-p)\lambda\mathbf1_{E^c},
\]
so $\|k_{\theta}\|_\infty=(2-p)\lambda$. If $0\le\theta<1$, choose $p=1-\theta$. Then $\Lambda_1=\theta\lambda$, $s_{\theta}=\infty$, and
\[
 \|k_{\theta}\|_\infty=(1+\theta)\lambda.
\]
Thus the coefficient $1+\theta$ in Proposition~\ref{prop:theta-family} is sharp for $0\le\theta<1$. For $\theta=1$, the same example gives $\|k_{\theta}\|_\infty=(2-p)\lambda\to2\lambda$ as $p\downarrow0$, so the coefficient $2$ cannot be improved.

Likewise,
\[
 k^{\mathrm{st}}_{\theta}=\lambda\mathbf1_E,
 \qquad
 k^{\mathrm{pr}}_{\theta}=(1-p)\lambda.
\]
Hence $\|k^{\mathrm{st}}_{\theta}\|_\infty=\lambda$, and for $0\le\theta<1$ choosing $p=1-\theta$ gives $\|k^{\mathrm{pr}}_{\theta}\|_\infty=\theta\lambda$. Thus the individual bounds for $k^{\mathrm{st}}_{\theta}$ and $k^{\mathrm{pr}}_{\theta}$ are also sharp in this range; at $\theta=1$ one again obtains asymptotic sharpness.
\end{example}

The preceding examples only concern the stopping-time decomposition. The next proposition is global: on the two-point filtration it establishes a lower-bound dichotomy for arbitrary decompositions we consider.

\begin{proposition}
\label{prop:global-joint-sharpness}
Let $\Omega=E\sqcup E^c$ with $\mathbb P(E)=p\in(0,\tfrac12]$, and define
\[
 \mathcal F_1=\{\varnothing,\Omega\},
 \qquad
 \mathcal F_n=\mathcal F_\infty=\{\varnothing,E,E^c,\Omega\},\qquad n\ge2.
\]
Fix $\lambda>0$ and let
\[
 f:=\frac{\lambda}{p}\mathbf1_E,
 \qquad
 \|f\|_1=\lambda.
\]
Assume that
\[
 f=g+h+k
\]
with $g,h,k\in L^1(\mathcal F_\infty)$ and
\[
 0\le k\le \beta\lambda\qquad\text{a.s.}
\]
for some $\beta\in[0,p^{-1}]$. Then one of the following alternatives holds:

\smallskip
\noindent\textup{(i)}
\[
 \mathbb P\Bigl(\sup_{n\ge1}|\mathbb E_n[g]|>0\Bigr)=1;
\]

\noindent\textup{(ii)} $g=0$ almost surely and
\begin{equation}
\label{eq:global-joint-piecewise}
 \Bigl\|\sum_{n=1}^{\infty}|\mathbb E_n[h]-\mathbb E_{n-1}[h]|\Bigr\|_1
 \ge
 \begin{cases}
 (3-\beta-2p)\lambda, & 0\le \beta\le1,\\[4pt]
 2(1-p\beta)\lambda, & 1\le \beta\le p^{-1}.
 \end{cases}
\end{equation}
Both lower bounds are sharp.
\end{proposition}

\begin{proof}
Write $q:=1-p$. Every $u\in L^1(\mathcal F_\infty)$ takes only two values, and hence
\[
 u=u_E\mathbf1_E+u_{E^c}\mathbf1_{E^c}.
\]
Since the filtration stabilizes at $n = 2$,
\[
 \mathbb E_1[u]=\mathbb E[u]=pu_E+qu_{E^c},
 \qquad
 \mathbb E_n[u]=u\quad(n\ge2).
\]
Therefore,
\[
 \Bigl\{\sup_{n\ge1}|\mathbb E_n[u]|>0\Bigr\}
 =\{u\ne0\}\cup\{\mathbb E[u]\ne0\}.
\]
If $\mathbb E[u]\ne0$, then
\[
\mathbb P\Bigl(\sup_{n\ge1}|\mathbb E_n[u]|>0\Bigr)=1.
\]
If $\mathbb E[u] = 0$, then $p u_E + q u_{E^c} = 0$. Hence either $u_E=u_{E^c}=0$,
in which case $u=0$ a.s., or $u_E$ and $u_{E^c}$ have opposite signs, so
$u\ne0$ on both $E$ and $E^c$, and again
\[
\mathbb P\Bigl(\sup_{n\ge1}|\mathbb E_n[u]|>0\Bigr)=1.
\]
Applying this to $u=g$ shows that
\[
 \mathbb P\Bigl(\sup_{n\ge1}|\mathbb E_n[g]|>0\Bigr)<1
 \quad\Longrightarrow\quad
 g=0\quad\text{a.s.}
\]
Thus alternative \textup{(i)} fails only when $g=0$.

Assume from now on that $g=0$. Write
\[
 k=a\mathbf1_E+b\mathbf1_{E^c},
 \qquad 0\le a,b\le\beta\lambda.
\]
Then
\[
 h=\Bigl(\frac{\lambda}{p}-a\Bigr)\mathbf1_E-b\mathbf1_{E^c}.
\]
Since $\mathbb E_0\equiv0$, $\mathbb E_1[h]=\mathbb E[h]$, and $\mathbb E_n[h]=h$ for $n\ge2$,
\[
 \sum_{n=1}^{\infty}|\mathbb E_n[h]-\mathbb E_{n-1}[h]|
 =|\mathbb E[h]|+|h-\mathbb E[h]|.
\]
Set
\[
 x:=\frac{\lambda}{p}-a,
 \qquad
 y:=-b.
\]
Then $h=x\mathbf1_E+y\mathbf1_{E^c}$ and
\[
 \mathbb E[h]=px+qy=\lambda-pa-qb.
\]
Moreover,
\[
 x-\mathbb E[h]=q(x-y),
 \qquad
 y-\mathbb E[h]=-p(x-y),
\]
so
\[
 \mathbb E [|h-\mathbb E[h]|]
 =pq|x-y|+qp|x-y|=2pq|x-y|.
\]
Therefore,
\begin{equation}
\label{eq:Vab-formula}
 \Bigl\|\sum_{n=1}^{\infty}|\mathbb E_n[h]-\mathbb E_{n-1}[h]|\Bigr\|_1
 =|\lambda-pa-qb|+2pq\Bigl|\frac{\lambda}{p}-a+b\Bigr|.
\end{equation}
Since $0\le a,b\le\beta\lambda$ and $\beta\le p^{-1}$,
\[
 \frac{\lambda}{p}-a+b\ge \lambda\Bigl(\frac1p-\beta\Bigr)\ge0,
\]
so the second absolute value in \eqref{eq:Vab-formula} can be dropped.

If $0\le\beta\le1$, then
\[
 \lambda-pa-qb\ge \lambda-\beta\lambda\ge0,
\]
so both absolute values can be removed and
\[
 \Bigl\|\sum_{n=1}^{\infty}|\mathbb E_n[h]-\mathbb E_{n-1}[h]|\Bigr\|_1
 =\lambda-pa-qb+2pq\Bigl(\frac{\lambda}{p}-a+b\Bigr).
\]
Because $p\le\tfrac12$, the right-hand side is decreasing in both $a$ and $b$, hence minimized at $a=b=\beta\lambda$. This gives
\[
 \Bigl\|\sum_{n=1}^{\infty}|\mathbb E_n[h]-\mathbb E_{n-1}[h]|\Bigr\|_1
 \ge(1-\beta)\lambda+2(1-p)\lambda
 =(3-\beta-2p)\lambda.
\]
Equality holds for
\[
 g=0,
 \qquad
 k=\beta\lambda\mathbf1_\Omega,
 \qquad
 h=f-k.
\]

Assume next that $1\le\beta\le p^{-1}$. Define
\[
 D_+:=\{(a,b):0\le a,b\le\beta\lambda,\ pa+qb\le\lambda\},
\]
\[
 D_-:=\{(a,b):0\le a,b\le\beta\lambda,\ pa+qb\ge\lambda\}.
\]
If $(a,b)\in D_+$, then
\[
 \Phi(a,b):=|\lambda-pa-qb|+2pq\Bigl(\frac{\lambda}{p}-a+b\Bigr)
 =\lambda-pa-qb+2pq\Bigl(\frac{\lambda}{p}-a+b\Bigr).
\]
For fixed $a$, the map $b\mapsto\Phi(a,b)$ is decreasing on the corresponding slice of $D_+$, so its minimum is attained on the boundary $pa+qb=\lambda$. If $(a,b)\in D_-$, then
\[
 \Phi(a,b)=pa+qb-\lambda+2pq\Bigl(\frac{\lambda}{p}-a+b\Bigr),
\]
and for fixed $a$ the map $b\mapsto\Phi(a,b)$ is increasing on the corresponding slice of $D_-$. Hence the global minimum is attained on the same boundary line $pa+qb=\lambda$.

Along this boundary one has
\[
 \Phi(a,b)=2pq\Bigl(\frac{\lambda}{p}-a+b\Bigr)=2(\lambda-pa),
\]
which is decreasing in $a$. Since $a\le\beta\lambda$, the minimum is attained at
\[
 a=\beta\lambda,
 \qquad
 b=\frac{\lambda-p\beta\lambda}{q}.
\]
This choice is admissible because $\beta\ge1$ and $\beta\le p^{-1}$. Therefore,
\[
 \Bigl\|\sum_{n=1}^{\infty}|\mathbb E_n[h]-\mathbb E_{n-1}[h]|\Bigr\|_1
 \ge2(\lambda-p\beta\lambda)
 =2(1-p\beta)\lambda.
\]
Equality holds for
\[
 g=0,
 \qquad
 k=\beta\lambda\mathbf1_E+\frac{1-p\beta}{1-p}\lambda\mathbf1_{E^c},
 \qquad
 h=f-k.
\]
This proves \eqref{eq:global-joint-piecewise}.
\end{proof}

\begin{corollary}
\label{cor:global-joint-sharpness}
In the setting of Proposition~\ref{prop:global-joint-sharpness}, take $\beta=1$. Then every decomposition
\[
 f=g+h+k,
 \qquad 0\le k\le\lambda,
\]
satisfies
\[
 \mathbb P\Bigl(\sup_{n\ge1}|\mathbb E_n[g]|>0\Bigr)=1
 \qquad\text{or}\qquad
 \Bigl\|\sum_{n=1}^{\infty}|\mathbb E_n[h]-\mathbb E_{n-1}[h]|\Bigr\|_1
 \ge2(1-p)\|f\|_1.
\]
Consequently, no pair of constants $c_1<1$ and $c_2<2$ can simultaneously replace the coefficients in Theorem~\ref{thm:stein-gundy}\textup{(a)--(b)} under the same condition $0\le k\le\lambda$.
\end{corollary}

\begin{proof}
The first assertion is Proposition~\ref{prop:global-joint-sharpness} with $\beta=1$ and $\|f\|_1=\lambda$. Now let $c_1<1$ and $c_2<2$. Choose $p\in(0,\tfrac12]$ so small that $c_2<2(1-p)$. If there were a decomposition such that
\[
 \mathbb P\Bigl(\sup_{n\ge1}|\mathbb E_n[g]|>0\Bigr)
 \le c_1\frac{\|f\|_1}{\lambda}=c_1<1
\]
and
\[
 \Bigl\|\sum_{n=1}^{\infty}|\mathbb E_n[h]-\mathbb E_{n-1}[h]|\Bigr\|_1
 \le c_2\|f\|_1,
\]
then Proposition~\ref{prop:global-joint-sharpness} would force
\[
 c_2\|f\|_1\ge2(1-p)\|f\|_1,
\]
contrary to the choice of $p$.
\end{proof}

\begin{remark}
\label{rem:global-joint-sharpness}
Corollary~\ref{cor:global-joint-sharpness} is a global statement concerning all decompositions $f=g+h+k$ on the two-point filtration. While one may suppress
the localized term at the cost of enlarging $h$, the two coefficients cannot be improved simultaneously under the constraint $0\le k\le\lambda$.

The same proposition also applies to the parametric condition $0\le k\le(1+\theta)\lambda$: one simply takes $\beta=1+\theta$. Thus, whenever $0<p\le(1+\theta)^{-1}$ and
\[
 \mathbb P\Bigl(\sup_{n\ge1}|\mathbb E_n[g]|>0\Bigr)<1,
\]
one has
\[
 \Bigl\|\sum_{n=1}^{\infty}|\mathbb E_n[h]-\mathbb E_{n-1}[h]|\Bigr\|_1
 \ge2\bigl(1-(1+\theta)p\bigr) \, \|f\|_1.
\]
In particular,
\[
 \liminf_{p\downarrow0}\frac1{\|f\|_1}
 \Bigl\|\sum_{n=1}^{\infty}|\mathbb E_n[h]-\mathbb E_{n-1}[h]|\Bigr\|_1
 \ge2.
\]
Thus enlarging the $L^\infty$ allowance on the bounded term by any fixed factor does not remove the asymptotic cost $2\|f\|_1$ in the $h$ term when one insists on eliminating the localized part.
\end{remark}

\section{A weak-type $(1,1)$ estimate for truncated martingale multipliers}
\label{sec:multipliers}

Let $a=(a_n)_{n\ge1}\in\ell^\infty$. For $N\ge1$ and $f\in L^1(\Omega,\mathbb P)$, define the truncated martingale multiplier
\[
 T_N(a;f):=\sum_{n=1}^{N}a_n\,d f_n,
 \qquad
 d f_n:=\mathbb E_n[f]-\mathbb E_{n-1}[f].
\]

\begin{theorem}
\label{thm:weak11}
For every $N\ge1$, every $f\in L^1(\Omega,\mathbb P)$, and every $\lambda>0$,
\[
 \mathbb P\bigl(|T_N(a;f)|>\lambda\bigr)
 \le \frac{16\|a\|_{\ell^\infty}\|f\|_1}{\lambda}.
\]
\end{theorem}

\begin{proof}
By homogeneity, it is enough to assume $\|a\|_{\ell^\infty}=1$. Replacing $f$ by $\mathbb E[f\mid\mathcal F_\infty]$ does not change the martingale differences and does not increase the $L^1$ norm, so we may assume $f\in L^1(\mathcal F_\infty)$.

Apply Corollary~\ref{cor:stein-gundy} to $f$ at level $\lambda/2$, obtaining
\[
 f=g+h+k.
\]
Then
\begin{equation}
\label{eq:split}
\begin{aligned}
\mathbb P\bigl(|T_N(a;f)|>\lambda\bigr)
&\le \mathbb P\bigl(|T_N(a;g)|>0\bigr)\\
&\quad +\mathbb P\bigl(|T_N(a;h)|>\lambda/2\bigr)
 +\mathbb P\bigl(|T_N(a;k)|>\lambda/2\bigr).
\end{aligned}
\end{equation}
If $\sup_{n\ge1}|\mathbb E_n[g]|=0$, then all martingale differences of $g$ vanish and therefore $T_N(a;g)=0$. Hence, by Corollary~\ref{cor:stein-gundy}\textup{(a$'$)},
\begin{equation}
\label{eq:term-g}
\mathbb P\bigl(|T_N(a;g)|>0\bigr)
\le \mathbb P\Bigl(\sup_{n\ge1}|\mathbb E_n[g]|>0\Bigr)
\le \frac{4\|f\|_1}{\lambda}.
\end{equation}
Next,
\[
 |T_N(a;h)|\le \sum_{n=1}^{N}|d h_n|\le \sum_{n=1}^{\infty}|d h_n|,
\]
so Markov's inequality and Corollary~\ref{cor:stein-gundy}\textup{(b$'$)} give
\begin{equation}
\label{eq:term-h}
\mathbb P\bigl(|T_N(a;h)|>\lambda/2\bigr)
\le \frac{2}{\lambda}\Bigl\|\sum_{n=1}^{\infty}|d h_n|\Bigr\|_1
\le \frac{8\|f\|_1}{\lambda}.
\end{equation}
Finally, since $k_n=\mathbb E_n[k]$ and martingale differences are orthogonal in $L^2$,
\[
 \|T_N(a;k)\|_2^2
 =\sum_{n=1}^{N}|a_n|^2\,\|\mathbb E_n[k]-\mathbb E_{n-1}[k]\|_2^2
 \le \|k\|_2^2.
\]
Therefore, by Markov's inequality and Corollary~\ref{cor:stein-gundy}\textup{(c$'$)},
\begin{equation}
\label{eq:term-k}
\mathbb P\bigl(|T_N(a;k)|>\lambda/2\bigr)
\le \frac{4\|T_N(a;k)\|_2^2}{\lambda^2}
\le \frac{4\|k\|_2^2}{\lambda^2}
\le \frac{4\|f\|_1}{\lambda}.
\end{equation}
Combining \eqref{eq:split}, \eqref{eq:term-g}, \eqref{eq:term-h}, and \eqref{eq:term-k} completes the proof.
\end{proof}

\begin{remark}
Standard Marcinkiewicz interpolation, together with the estimate $\|T_N(a;f)\|_2\le \|a\|_{\ell^\infty}\|f\|_2$, yields the usual $L^p$ boundedness of martingale transforms for $1<p<\infty$. We now adopt a different point of view and fix $f\in L^2(\Omega,\mathbb P)$.
Then $T_N (a;f)$ is the martingale transform of $f$ associated with $a$, and
\begin{equation}\label{eq:ito}
\|T_N (a;f)\|_2^2
=
\sum_{n=1}^{N} a_n^2\,\mathbb E\bigl[|d f_n|^2\bigr],
\end{equation}
which may be viewed as a discrete analogue of It\^o's isometry. Interpolating Theorem~\ref{thm:weak11} with \eqref{eq:ito}, and viewing $T_N (a;f)$ as an operator acting on $a$, we obtain for $1<p\le2$,
\[
\|T_N (a;f)\|_p
\le
C_p\,
\|f\|_1^{\frac{2}{p}-1}
\Big(\sum_{n=1}^{N} |a_n|^{p'}\, \mathbb E\bigl[|d f_n|^2\bigr]\Big)^{1/p'}.
\]
Unlike the usual $L^p$ boundedness of martingale transforms, which depends on $\|f\|_p$, this estimate is expressed in terms of the quadratic energies $\mathbb E[|d f_n|^2]$.
\end{remark}

\section{A martingale John--Nirenberg inequality}
\label{sec:SGJN}

John--Nirenberg inequalities for martingale $\mathrm{BMO}$ are classical; see Garsia \cite[Ch.~III]{Garsia1973} and Kazamaki \cite{Kazamaki1994}. In this section we prove an explicit form on atomic $\alpha$-regular filtrations.

Throughout this section we assume that each $\mathcal F_n$ is generated by a countable measurable partition $\mathcal A_n$ of $\Omega$ into atoms, and that $\mathbb P(A)>0$ for every $A\in\mathcal A_n$.

For $A\in\mathcal A_n$ and $f\in L^1(\Omega,\mathbb P)$, write
\[
 f_A:=\frac{1}{\mathbb P(A)}\,\mathbb E[f \, \mathbf1_A].
\]
Then
\[
 f_n=\mathbb E_n[f]=\sum_{A\in\mathcal A_n}f_A\mathbf1_A.
\]

\begin{definition}
\label{def:BMO-mg}
For $f\in L^1(\mathcal F_\infty)$, define
\begin{equation}
\label{eq:BMO}
 \|f\|_{\mathrm{BMO}}
 :=\sup_{n\ge1}\sup_{A\in\mathcal A_n}
 \frac{1}{\mathbb P(A)}\,\mathbb E\bigl[|f-f_A|\, \mathbf1_A\bigr]
 =\sup_{n\ge1}\bigl\|\mathbb E_n[|f-f_n|]\bigr\|_\infty.
\end{equation}
We write $f\in\mathrm{BMO}(\Omega,\mathcal F_\bullet)$ if $\|f\|_{\mathrm{BMO}}<\infty$.
\end{definition}

\begin{definition}
\label{def:regular-filtration}
An atomic filtration $(\mathcal F_n)_{n\ge1}$ is called $\alpha$-regular if there exists $\alpha\in(0,1]$ such that for every $n\ge2$, every parent atom $P\in\mathcal A_{n-1}$, and every child atom $C\in\mathcal A_n$ with $C\subseteq P$,
\begin{equation}
\label{eq:alpha-regular}
 \mathbb P(C)\ge \alpha\,\mathbb P(P).
\end{equation}
\end{definition}

\begin{lemma}
\label{lem:overshoot}
Assume that $(\mathcal F_n)_{n\ge1}$ is an $\alpha$-regular atomic filtration. Let $f\in L^1(\Omega,\mathbb P)$ with $f\ge0$, and set $f_n=\mathbb E_n[f]$. Then for every $n\ge2$,
\begin{equation}
\label{eq:overshoot}
 f_n\le \alpha^{-1}f_{n-1}
 \qquad\text{a.s.}
\end{equation}
\end{lemma}

\begin{proof}
Fix $n\ge2$ and let $\omega\in C\subseteq P$, where $P\in\mathcal A_{n-1}$ and $C\in\mathcal A_n$. Then
\[
 f_n(\omega)=\frac1{\mathbb P(C)}\mathbb E[f \, \mathbf1_C]
 \le \frac{\alpha^{-1}}{\mathbb P(P)}\mathbb E[f \, \mathbf1_P]
 =\alpha^{-1}f_{n-1}(\omega).
 \qedhere
\]
\end{proof}

\begin{lemma}
\label{lem:martingale-CZ}
Assume that $(\mathcal F_n)_{n\ge1}$ is an $\alpha$-regular atomic filtration. Fix $N\ge1$ and an atom $A\in\mathcal A_N$. Let $g\in L^1(\mathcal F_\infty)$ and $\lambda>0$. If
\begin{equation}
\label{eq:mgCZ-parent-good}
 \frac{1}{\mathbb P(A)}\,\mathbb E[|g| \, \mathbf1_A]\le \lambda,
\end{equation}
then there exists a finite or countable collection of pairwise disjoint atoms $\{A_j\}_{j\in J}$ contained in $A$ such that

\smallskip
\noindent\textup{(i)}
\begin{equation}
\label{eq:mgCZ-avg}
 \lambda<\frac{1}{\mathbb P(A_j)}\,\mathbb E[|g| \, \mathbf1_{A_j}]\le \alpha^{-1}\lambda,
 \qquad j\in J;
\end{equation}

\noindent\textup{(ii)}
\begin{equation}
\label{eq:mgCZ-good}
 |g(\omega)|\le \lambda
 \qquad\text{for a.e. }\omega\in A\setminus\bigcup_{j\in J}A_j;
\end{equation}

\noindent\textup{(iii)}
\begin{equation}
\label{eq:mgCZ-measure}
 \sum_{j\in J}\mathbb P(A_j)
 \le \frac1{\lambda}\,\mathbb E[|g|\, \mathbf1_A].
\end{equation}
\end{lemma}

\begin{proof}
Define the nonnegative martingale
\[
 X_m:=\mathbb E_m[|g| \, \mathbf1_A],\qquad m\ge N,
\]
with terminal value $X_\infty:=|g|\mathbf1_A$. Let
\[
 r:=\inf\{m\ge N+1:\ X_m>\lambda\}.
\]
For $m\ge N+1$, set $B_m:=\{r=m\}\in\mathcal F_m$. Since the filtration is atomic, each $B_m\cap A$ is a union of atoms in $\mathcal A_m$ contained in $A$. Collecting these atoms over all $m$ gives a finite or countable family $\{A_j\}_{j\in J}$ of pairwise disjoint atoms such that
\[
 \bigcup_{j\in J}A_j=A\cap\{r<\infty\}.
\]
On $\{r<\infty\}$ one has $X_r>\lambda$, so
\[
 \lambda\,\mathbb P(r<\infty)\le \mathbb E[X_r].
\]
Lemma~\ref{lem:optional}, applied to the closed martingale $(X_m)_{m\ge N}$, yields
\[
 \mathbb E[X_r]=\mathbb E[X_\infty]=\mathbb E[|g| \, \mathbf1_A].
\]
Since $X_m=0$ on $A^c$, one has $r=\infty$ on $A^c$, hence
\[
 \mathbb P(r<\infty)=\sum_{j\in J}\mathbb P(A_j),
\]
which proves \eqref{eq:mgCZ-measure}.

If $\omega\in A$ and $r(\omega)=\infty$, then $X_m(\omega)\le\lambda$ for every $m\ge N$. As $X_m\to X_\infty=|g|\mathbf1_A$ almost surely, this yields \eqref{eq:mgCZ-good}.

Finally, let $A_j\in\mathcal A_m$ with $m\ge N+1$. Then $r=m$ on $A_j$, so $X_m>\lambda$ on $A_j$. Moreover, $X_{m-1}\le\lambda$ on the parent atom in $\mathcal A_{m-1}$ containing $A_j$: for $m\ge N+2$ this follows from the definition of $r$, while for $m=N+1$ it follows from \eqref{eq:mgCZ-parent-good}. Lemma~\ref{lem:overshoot} gives $X_m\le\alpha^{-1}X_{m-1}\le\alpha^{-1}\lambda$ on $A_j$. Since $X_m$ is constant on $A_j$ and equals $\mathbb P(A_j)^{-1}\mathbb E[|g|\mathbf1_{A_j}]$, this proves \eqref{eq:mgCZ-avg}.
\end{proof}

\begin{theorem}
\label{thm:SGJN}
Assume that $(\mathcal F_n)_{n\ge1}$ is an $\alpha$-regular atomic filtration. Let $f\in L^1(\mathcal F_\infty)$ with $\|f\|_{\mathrm{BMO}}<\infty$. Set
\[
 c_1:=e,
 \qquad
 c_2:=\frac{\alpha}{e}.
\]
Then for every $N\ge1$, every atom $A\in\mathcal A_N$, and every $t>0$,
\begin{equation}
\label{eq:SGJN-tail}
 \mathbb P\bigl(\{\omega\in A:\ |f(\omega)-f_A|>t\}\bigr)
 \le c_1\exp\!\left(-c_2\frac{t}{\|f\|_{\mathrm{BMO}}}\right)\mathbb P(A).
\end{equation}
\end{theorem}

\begin{proof}
Fix $A\in\mathcal A_N$ and write $B:=\|f\|_{\mathrm{BMO}}$. The case $B=0$ is trivial, so assume $B>0$. Let $s>B$ be a parameter.

Apply Lemma~\ref{lem:martingale-CZ} to $(f-f_A)\mathbf1_A$ at level $\lambda=s$. Since
\[
 \frac1{\mathbb P(A)}\mathbb E[|f-f_A| \, \mathbf1_A]\le B<s,
\]
we obtain pairwise disjoint atoms $\{A_{i_1}\}_{i_1\in\mathbb N}\subseteq A$ such that
\begin{align}
 s<\frac1{\mathbb P(A_{i_1})}\mathbb E[|f-f_A|\mathbf1_{A_{i_1}}]\le \alpha^{-1}s,
\label{eq:gen1-avg}
\\
 |f(\omega)-f_A|\le s
 \qquad\text{for a.e. }\omega\in A\setminus\bigcup_{i_1}A_{i_1}.
\label{eq:gen1-good}
\end{align}
Moreover,
\begin{equation}
\label{eq:star1}
 \sum_{i_1}\mathbb P(A_{i_1})
 \le \frac1s\mathbb E[|f-f_A|\, \mathbf1_A]
 \le \frac{B}{s}\mathbb P(A).
\end{equation}
For each such atom,
\begin{equation}
\label{eq:avg-increment-1}
 |f_{A_{i_1}}-f_A|
 \le \frac1{\mathbb P(A_{i_1})}\mathbb E[|f-f_A| \, \mathbf1_{A_{i_1}}]
 \le \alpha^{-1}s.
\end{equation}

Proceed inductively. Suppose that for some $k\ge1$ we have constructed pairwise disjoint atoms $\{A_{i_1,\dots,i_k}\}$ inside $A$, and set
\[
 E_k:=\bigcup_{i_1,\dots,i_k}A_{i_1,\dots,i_k}.
\]
For each parent $P=A_{i_1,\dots,i_k}$, apply Lemma~\ref{lem:martingale-CZ} to $(f-f_P)\mathbf1_P$ at level $s$. The defining property of $\|f\|_{\mathrm{BMO}}$ implies
\[
 \frac1{\mathbb P(P)}\mathbb E[|f-f_P| \, \mathbf1_P]\le B<s.
\]
Summing \eqref{eq:mgCZ-measure} over all parents yields
\begin{equation}
\label{eq:measure-k}
 \mathbb P(E_k)\le \left(\frac{B}{s}\right)^k\mathbb P(A),
 \qquad k\ge1.
\end{equation}

Fix $k\ge1$ and let $\omega\in A\setminus E_k$. Construct nested atoms
\[
 A=P_0\supseteq P_1\supseteq\cdots\supseteq P_\ell
\]
by letting $P_j$ be the $j$th generation atom containing $\omega$ while $\omega\in E_j$, and stop when $\omega\notin E_{\ell+1}$. Then $\ell\le k-1$. Using \eqref{eq:gen1-good} at the last stage and \eqref{eq:avg-increment-1} at each previous stage, we obtain
\begin{equation}
\label{eq:pointwise-k}
\begin{split}
 |f(\omega)-f_A|
 &\le |f(\omega)-f_{P_\ell}|+\sum_{j=1}^{\ell}|f_{P_j}-f_{P_{j-1}}|\\
 &\le s+\ell\alpha^{-1}s
 \le (\ell+1)\alpha^{-1}s
 \le k\alpha^{-1}s.
\end{split}
\end{equation}

Let $t>0$. If $t\ge \alpha^{-1}s$, choose $k\in\mathbb N$ such that
\begin{equation}
\label{eq:choose-k}
 k\alpha^{-1}s\le t<(k+1)\alpha^{-1}s.
\end{equation}
Then \eqref{eq:pointwise-k} implies
\[
 \{\omega\in A:\ |f(\omega)-f_A|>t\}\subseteq E_k,
\]
so that by \eqref{eq:measure-k},
\begin{equation}
\label{eq:tail-k}
 \mathbb P(\{\omega\in A:\ |f(\omega)-f_A|>t\})
 \le \left(\frac{B}{s}\right)^k\mathbb P(A)
 =\exp\!\left(-k\log\frac{s}{B}\right)\mathbb P(A).
\end{equation}
Since $t<(k+1)\alpha^{-1}s$, one has $k>\alpha t/s-1$, and therefore
\begin{equation}
\label{eq:tail-parameter}
 \mathbb P(\{\omega\in A:\ |f(\omega)-f_A|>t\})
 \le \frac{s}{B}\exp\!\left(-\frac{\alpha t}{s}\log\frac{s}{B}\right)\mathbb P(A).
\end{equation}
For $0<t<\alpha^{-1}s$, the trivial estimate by $\mathbb P(A)$ is still dominated by the right-hand side of \eqref{eq:tail-parameter}. Finally choose $s=eB$. Then $s/B=e$ and $\log(s/B)=1$, so \eqref{eq:tail-parameter} becomes exactly \eqref{eq:SGJN-tail}.
\end{proof}

\begin{remark}
\label{rem:SGJN-parameter}
The proof yields a one-parameter family of tail bounds. If one chooses $s=u\|f\|_{\mathrm{BMO}}$ with any $u>1$ in \eqref{eq:tail-parameter}, then
\[
 \mathbb P(\{\omega\in A:\ |f(\omega)-f_A|>t\})
 \le u\exp\!\left(-\frac{\alpha\log u}{u}\frac{t}{\|f\|_{\mathrm{BMO}}}\right)\mathbb P(A).
\]
The factor $\alpha\log u/u$ is maximized at $u=e$, which recovers Theorem~\ref{thm:SGJN}.
\end{remark}

\begin{corollary}
\label{cor:exp-int}
Assume the hypotheses of Theorem~\ref{thm:SGJN}. For every $N\ge1$, every atom $A\in\mathcal A_N$, and every $\beta\in(0,\alpha/e)$,
\begin{equation}
\label{eq:exp-int}
 \frac1{\mathbb P(A)}\,\mathbb E\left[\exp\!\left(\frac{\beta |f-f_A|}{\|f\|_{\mathrm{BMO}}}\right)\mathbf1_A\right]
 \le 1+\frac{e^2\beta}{\alpha-e\beta}.
\end{equation}
\end{corollary}

\begin{proof}
Write $B:=\|f\|_{\mathrm{BMO}}$. The case $B=0$ is trivial. Assume $B>0$ and set $Y:=|f-f_A|$ on $A$. By Theorem~\ref{thm:SGJN},
\[
 \mathbb P(\{\omega\in A:\ Y>t\})
 \le e\exp\!\left(-\frac{\alpha}{e}\frac{t}{B}\right)\mathbb P(A),
 \qquad t>0.
\]
For $K>0$, the layer-cake identity gives
\[
 \mathbb E\bigl[(e^{K Y}-1)\mathbf1_A\bigr]
 =\int_0^{\infty}K e^{K t}\,\mathbb P(\{\omega\in A:\ Y>t\})\,dt.
\]
Choose $K=\beta/B$. Then
\begin{align*}
 \mathbb E\left[e^{\frac{\beta Y}{B}}\mathbf1_A\right]
 &\le \mathbb P(A)
 +\frac{e\beta}{B}\,\mathbb P(A)
 \int_0^{\infty}\exp\!\left[-\left(\frac{\alpha}{e}-\beta\right)\frac{t}{B}\right]dt\\
 &=\left(1+\frac{e^2\beta}{\alpha-e\beta}\right)\mathbb P(A).
\end{align*}
Dividing by $\mathbb P(A)$ gives \eqref{eq:exp-int}.
\end{proof}

\begin{remark}
In the probabilistic literature one often defines martingale $\mathrm{BMO}$ by a stopping-time norm such as
\[
 \sup_T\Bigl\|\mathbb E\bigl[|M_\infty-M_T|^p\mid\mathcal F_T\bigr]^{1/p}\Bigr\|_\infty<\infty,
\]
and proves conditional exponential integrability for $M_\infty-M_T$; see, for instance, Kazamaki~\cite{Kazamaki1994}. For discrete martingales, Garsia~\cite[Ch.~III]{Garsia1973} established a John--Nirenberg inequality without any regularity assumption on the filtration, using a different $\mathrm{BMO}$ seminorm. The estimate obtained here is of a different kind: it is proved for atomic $\alpha$-regular filtrations and all constants are explicit.
\end{remark}

\end{document}